\documentclass[11pt,letterpaper]{article}
\usepackage[margin=1.05in]{geometry}
\usepackage{mathtools,amssymb,amsthm}
\usepackage{microtype}
\usepackage[
 hidelinks,
 pdftitle={A sharp higher-order Cheeger inequality},
 pdfauthor={Yunpeng Li},
 pdfsubject={Spectral graph theory and higher-order Cheeger inequalities},
 pdfkeywords={higher-order Cheeger inequality, spectral partitioning, minimum cut}
]{hyperref}

\title{A sharp higher-order Cheeger inequality}
\author{Yunpeng Li\thanks{School of Data Science, The Chinese University
of Hong Kong, Shenzhen. Email: \texttt{liyunpeng@cuhk.edu.cn}.}}
\date{}

\numberwithin{equation}{section}
\newtheorem{theorem}{Theorem}[section]
\newtheorem{lemma}[theorem]{Lemma}
\newtheorem{proposition}[theorem]{Proposition}

\theoremstyle{definition}
\newtheorem{definition}[theorem]{Definition}
\theoremstyle{remark}

\newcommand{\R}{\mathbb R}
\newcommand{\Ex}{\mathbb E}
\newcommand{\Prob}{\mathbb P}
\newcommand{\one}{\mathbf 1}
\newcommand{\En}{\mathcal E}
\newcommand{\norm}[1]{\lVert #1\rVert}
\newcommand{\ip}[2]{\langle #1,#2\rangle}
\DeclareMathOperator{\tr}{tr}
\DeclareMathOperator{\rank}{rank}
\DeclareMathOperator{\ran}{ran}
\DeclareMathOperator{\vol}{vol}
\DeclareMathOperator{\supp}{supp}
\DeclareMathOperator{\cut}{cut}
\DeclareMathOperator{\TV}{TV}
\newcommand{\tail}{\operatorname{tail}_1}

\begin{document}
\maketitle

\begin{abstract}
Let $\lambda_k(G)$ be the $k$th eigenvalue of the normalized Laplacian
of a finite undirected weighted graph $G$ with positive degrees, where
$k$ is an integer satisfying $1\le k\le |V(G)|$. Let $\varphi_k(G)$ be the
minimum possible maximum conductance of $k$ disjoint nonempty vertex
sets. We prove
\[
                  \varphi_k(G)\le C\sqrt{\lambda_k(G)\log(k+1)}
\]
for an absolute constant $C$. The number of sets and the spectral index
are both $k$, and conductance is measured in the original graph.
The logarithmic dependence is optimal up to an absolute constant.
The proof combines geometric partitioning of the spectral embedding
with minimum-cut improvement and adaptive projections in coefficient
space. A single conductance threshold is used throughout the construction.
The resulting maps have disjoint supports, and the sum of their Gram
matrices is bounded below by an absolute positive multiple of the identity.
A dyadic maximal estimate bounds the sum of their internal energies
uniformly over unit coefficient vectors. A dimension argument using
local eigenvalues then yields exactly $k$ disjoint sparse cuts.
\end{abstract}

\section{Introduction}\label{sec:introduction}

Let $G=(V,w)$ be a finite undirected weighted graph, with symmetric
nonnegative weights $w_{uv}$ and positive degrees $d_v=\sum_u w_{uv}$.
For disjoint sets $A,B\subseteq V$, write
$w(A,B)=\sum_{u\in A,v\in B}w_{uv}$ and $w(v,B)=w(\{v\},B)$.
For $S\subseteq V$, write $S^c=V\setminus S$ and
$\vol(S)=\sum_{v\in S}d_v$. We also write
$\cut(S)=w(S,V\setminus S)$. For a nonempty
set $S$, its edge conductance is
\[
                       \varphi(S)=\frac{w(S,V\setminus S)}{\vol(S)}.
\]
We use this one-sided definition even when $S$ has more than half the
total volume. Define
the $k$-way edge conductance by
\begin{equation}\label{eq:phik}
 \varphi_k(G)=
 \min_{\substack{\varnothing\ne S_1,\ldots,S_k\subseteq V\\
                    S_i\cap S_j=\varnothing\ (i\ne j)}}
                    \max_{1\le i\le k}\varphi(S_i).
\end{equation}
The sets need not cover $V$. Let
$0=\lambda_1(G)\le\cdots\le\lambda_{|V|}(G)$ be the eigenvalues of
the normalized Laplacian. We omit $G$ from $\lambda_j(G)$ when
the graph is clear.
For a nonzero real function $f$ on $V$, its Rayleigh quotient is
\[
 \frac{\sum_{\{u,v\}}w_{uv}(f(u)-f(v))^2}{\sum_vd_vf(v)^2},
\]
where the numerator counts each unordered nonloop edge once.

Lee, Oveis Gharan, and Trevisan~\cite{LOT} proved
\[
              \frac{\lambda_k(G)}2\le\varphi_k(G)
                         \le O(k^2)\sqrt{\lambda_k(G)}.
\]
The dependence on $k$ improves when fewer sets are required.
Louis, Raghavendra, Tetali, and Vempala~\cite{LRTV} and
Lee, Oveis Gharan, and Trevisan~\cite{LOT} obtain a constant fraction
of $k$ disjoint sets with conductance $O(\sqrt{\lambda_k\log k})$.
In particular, the latter work gives
$\varphi_k(G)\le O(\sqrt{\lambda_{2k}\log k})$ when $2k\le |V|$.
The question is whether the same logarithmic dependence holds when
the number of sets and the eigenvalue index are both $k$.

Makarychev, Makarychev, Shan, and Vijayaraghavan~\cite{MMSV}
prove related inequalities for partitions with buffers, where edges
to small buffer sets are excluded from the boundary. Their logarithmic
estimate also allows a larger eigenvalue index.

Lau's book manuscript~\cite[Theorem~7.17]{Lau}
presents the bound
\begin{equation}\label{eq:lau}
              \varphi_k(G)\le O((\log k)^{3/2})\sqrt{\lambda_k(G)}.
\end{equation}
Its proof repeatedly improves the current family of sparse cuts.
Each round reduces the deficit in the number of sets by a constant
factor and adds $O(\sqrt{\lambda_k\log k})$ to the conductance bound.
After $O(\log k)$ rounds, the family contains $k$ sets. In an earlier
paper~\cite{Li}, the author obtained a weaker polylogarithmic bound
through localization of the spectral embedding and adaptive coefficient
projections.
That method also gives disjoint functions with small original-graph
Rayleigh quotients.

Here we combine cut improvement with adaptive spectral localization
to improve the bound in \eqref{eq:lau} by a factor of $\log k$.
The number of sets and the eigenvalue index remain equal to $k$.

\begin{theorem}\label{thm:main}
There is an absolute constant $C$ such that, for every finite undirected
weighted graph $G=(V,w)$ with positive degrees and every integer
$1\le k\le |V|$, there exist $k$ pairwise disjoint nonempty sets
$S_1,\ldots,S_k$ satisfying
\begin{equation}\label{eq:main}
              \max_{1\le i\le k}\varphi(S_i)
                      \le C\sqrt{\lambda_k(G)\log(k+1)}.
\end{equation}
\end{theorem}

The exponent $1/2$ on the logarithm cannot be decreased. The noisy
hypercube example in~\cite[Section~4.4]{LOT} gives graphs for which
$\varphi_k(G)/\sqrt{\lambda_k(G)}\ge c\sqrt{\log k}$.
We verify the lower bound with spectral index $k$ in
Section~\ref{sec:sharpness}.

\subsection{Proof overview}

Let $F:\R^k\to L^2(d)$ have as columns
eigenfunctions for the first $k$ eigenvalues of the random-walk
Laplacian, orthonormal in the degree-weighted inner product.
Every nonzero function in its range has Rayleigh quotient at most $\lambda_k$.

\paragraph{A fixed conductance threshold.}
We use one threshold $\tau=C_0\sqrt{\lambda_k\log(k+1)}$ throughout,
where $C_0$ is a sufficiently large absolute constant.
For a vertex set $B$, choose $S\subseteq B$ to minimize
\[
                         \cut(T)-\tau\vol(T),\qquad T\subseteq B.
\]
If $S$ is nonempty, then $\varphi(S)\le\tau$.
Comparison with $S\setminus T$ also shows that every $T\subseteq S$ satisfies
\begin{equation}\label{eq:intro-stability}
                w(T,S^c)\le\tau\vol(T)+w(T,S\setminus T).
\end{equation}
This is the boundary inequality used in Lau's local-to-global lemma
for cores~\cite[Lemma~7.18]{Lau}, with $\tau$ in place of
$\varphi(S)$. For nonempty $T$, it implies
$\varphi(T)\le\tau+2\varphi_S(T)$, where
$\varphi_S(T)$ is local conductance with original degrees in the
denominator. The same minimization also controls how much spectral
mass is rejected.

\paragraph{Keeping the remaining directions.}
At each stage, we work in a subspace of coefficient space.
Geometric partitioning and cut improvement select disjoint sets from
the unassigned vertices. Vertices that are not selected remain
available. We then retain the spectral subspace of the remaining Gram matrix
corresponding to eigenvalues close to one. The corresponding spectral
functions have little mass on previously selected sets.
The inequality \eqref{eq:intro-stability} then controls the boundary
contributions from those sets without increasing $\tau$.

A trace estimate suffices to reduce the dimension of this subspace by
a constant factor. Because selection is by vertex sets, the Gram matrices on
the selected and remaining vertices add exactly. The discarded
coefficient spaces are orthogonal, so after $O(\log k)$ stages the sum
of the Gram matrices of the selected maps is bounded below by an
absolute positive multiple of the identity. Each individual Gram matrix
has a small sum of eigenvalues other than the largest.

\paragraph{Obtaining exactly $k$ sets.}
The selected maps use different, nested coefficient projections.
A dyadic maximal estimate bounds the sum of their internal energies,
uniformly over unit coefficient vectors, by
$O((1+\log\log(k+2))^2\lambda_k)$. A dimension argument then shows that
the local eigenvalue counts below a constant multiple of this bound,
each truncated at two, sum to at least $k$. Local Cheeger rounding gives
one or two sets from each selected vertex set. The boundary inequality
converts their local conductance to original conductance at cost
\[
                  \tau+O((1+\log\log(k+2))\sqrt{\lambda_k}),
\]
which is $O(\sqrt{\lambda_k\log(k+1)})$.

The proof follows this order. Section~\ref{sec:preliminaries} fixes
the local operators and spectral notation. Sections~\ref{sec:cut}
and~\ref{sec:geometry} establish cut improvement and geometric
partitioning. Section~\ref{sec:adaptive} constructs the selected
maps, and Section~\ref{sec:extraction} completes the proof.

\section{Notation and preliminaries}\label{sec:preliminaries}

Self-loops are allowed and count once in degree. Sums indexed by
$\{u,v\}$ count each unordered nonloop edge once.
Let $A=(w_{uv})$,
$D=\operatorname{diag}(d_v)$, and
\[
                     \mathcal L=D^{-1/2}(D-A)D^{-1/2}.
\]
We write $L^2(d)$ for the space of real functions on $V$ with inner product
$\ip{f}{g}_d=\sum_vd_vf(v)g(v)$, induced norm
$\norm f_d^2=\ip{f}{f}_d$, and energy
\[
                \En(f)=\sum_{\{u,v\}}w_{uv}(f(u)-f(v))^2.
\]
The random-walk Laplacian $\Delta=D^{-1}(D-A)$ is self-adjoint in
this inner product and similar to $\mathcal L$.
In this notation, the Rayleigh quotient of a nonzero function $f$ is
$\En(f)/\norm f_d^2$.
For vector-valued maps, energy is the sum of coordinate energies.
All logarithms are natural except $\log_2$. Constants denoted by $C$
are positive and absolute unless a dependence is indicated; their
values may increase between estimates.

\paragraph{Local conductance and local eigenvalues.}
For a nonempty $S\subseteq V$ and nonempty $T\subseteq S$, define
\[
 \varphi_S(T)=\frac{w(T,S\setminus T)}{\vol(T)},\qquad
 \En_S(f)=\sum_{\{u,v\}\subset S}w_{uv}(f(u)-f(v))^2.
\]
The local graph keeps the internal edges of $S$ and adds the outgoing
weight $w(v,S^c)$ to the loop at $v$. Thus its degrees are the original
$d_v$. Its eigenvalues, denoted by $\lambda_j(S)$, are those of
\[
  (\Delta_S f)(v)=\frac1{d_v}\sum_{u\in S}w_{vu}(f(v)-f(u))
\]
in $L^2(d|_S)$. Norms and inner products of functions on $S$ use
the restricted degree weights $d|_S$. In particular, $\lambda_1(S)=0$.
This is the weighted version of the local operator used
in~\cite[Section~7.7]{Lau}; it is different from the Dirichlet
operator, which includes boundary energy.

\paragraph{Spectral embedding and Gram matrices.}
Let $F:\R^k\to L^2(d)$ have as columns eigenfunctions for the first
$k$ eigenvalues of $\Delta$, orthonormal in the degree-weighted inner product.
Adjoints are taken with respect to the specified inner products;
in particular, $F^*=F^TD$.
Let $\mathbf f_v=F^Te_v\in\R^k$ be the transpose of the row of $F$
indexed by $v$, where $e_v$ is the corresponding coordinate vector.
The map $v\mapsto\mathbf f_v$ is the spectral embedding, while
$F:\R^k\to L^2(d)$ acts on coefficient vectors by
$(Fx)(v)=\ip{\mathbf f_v}{x}$. Then
\begin{equation}\label{eq:eigenfield}
                   F^*F=I_k,\qquad \En(Fx)\le\lambda_k\norm x^2.
\end{equation}
For a matrix $Z\in\R^{k\times\ell}$ with orthonormal columns,
$Z^TZ=I_\ell$, put $\mathbf y_v=Z^T\mathbf f_v$. The columns of $Z$
form a basis of the active coefficient space, and $FZ$ represents $F$
on that space. Then
\begin{equation}\label{eq:current-field}
 \sum_vd_v\mathbf y_v\mathbf y_v^T=I_\ell,\qquad
 \En(FZ)=\sum_{\{u,v\}}w_{uv}\norm{\mathbf y_u-\mathbf y_v}^2
                                                        \le\ell\lambda_k.
\end{equation}
For a set $B$, let $\one_B$ denote both its indicator function and
the operator of multiplication by that function. The matrix
$Z^TF^*\one_BFZ=\sum_{v\in B}d_v\mathbf y_v\mathbf y_v^T$ is the
weighted Gram matrix of the columns of $FZ$ restricted to $B$.
Matrix norms are operator norms and vector norms are Euclidean
unless indicated otherwise. For symmetric matrices, $M\preceq N$
means that $N-M$ is positive semidefinite.

For a positive semidefinite matrix $M$, we use $\tail(M)$ as
shorthand for the sum of its eigenvalues other than the largest. Thus
\begin{equation}\label{eq:tail}
 \tail(M)=\tr M-\lambda_{\max}(M)
            =\min_{\norm u=1}\tr\bigl((I-uu^T)M\bigr).
\end{equation}
The variational formula shows that $0\preceq M'\preceq M$ implies
$\tail(M')\le\tail(M)$.

\paragraph{Coarea and threshold rounding.}
For a real function $h$, define its total variation by
$\TV(h)=\sum_{\{u,v\}}w_{uv}|h(u)-h(v)|$.
For $h\ge0$, coarea gives
\begin{equation}\label{eq:coarea}
 \TV(h)=\int_0^\infty\cut(\{h>t\})\,dt,\qquad
 \sum_vd_vh(v)=\int_0^\infty\vol(\{h>t\})\,dt.
\end{equation}
For every real $f$, Cauchy--Schwarz gives
\begin{equation}\label{eq:variation}
 \TV(f^2)
 \le\sqrt{\En(f)\sum_{\{u,v\}}w_{uv}(f(u)+f(v))^2}
 \le\sqrt{2\En(f)\norm f_d^2}.
\end{equation}
Write $f_+=\max\{f,0\}$ and $f_-=\max\{-f,0\}$.
A nonzero nonnegative function $f$ has a nonempty superlevel set
$T\subseteq\supp f$ with $\varphi(T)\le\sqrt{2\En(f)}/\norm f_d$.
The same statement holds for the local graph on $S$, with $\En_S$
and $\varphi_S$. This is the Cheeger rounding step
in~\cite[Lemma~7.4]{Lau}.

\section{Cut improvement at a fixed threshold}\label{sec:cut}

We isolate the boundary property needed for local-to-global rounding.

\begin{definition}
For $\tau\ge0$, a nonempty set $S$ is \emph{$\tau$-stable} if
\begin{equation}\label{eq:stable}
            w(T,S^c)\le\tau\vol(T)+w(T,S\setminus T)
                                      \qquad\text{for every }T\subseteq S.
\end{equation}
\end{definition}

Taking $T=S$ shows that $\varphi(S)\le\tau$. More generally,
\begin{equation}\label{eq:local-global}
                       \varphi(T)\le\tau+2\varphi_S(T)
                              \qquad(\varnothing\ne T\subseteq S).
\end{equation}
A core of conductance at most $\tau$, in the terminology
of~\cite[Section~7.7]{Lau}, is $\tau$-stable by the proof of
Lemma~7.18 there. We need only \eqref{eq:stable}, rather than the full
minimum-conductance property of a core.

\begin{lemma}[Cut improvement]\label{lem:cut}
Fix $B\subseteq V$ and $\tau>0$. Let $S\subseteq B$, with the empty
set allowed, minimize
\begin{equation}\label{eq:cut-objective}
                     J(T)=\cut(T)-\tau\vol(T),\qquad T\subseteq B.
\end{equation}
If $S$ is nonempty, then it is $\tau$-stable. For every real function
$g$ supported in $B$,
\begin{equation}\label{eq:mass-loss}
 \tau\norm{\one_{B\setminus S}g}_d^2+
             \TV\bigl((\one_Sg)^2\bigr)\le\TV(g^2).
\end{equation}
\end{lemma}

\begin{proof}
For $T\subseteq S$, comparison with $S\setminus T$ gives
\[
 0\le J(S\setminus T)-J(S)
       =w(T,S\setminus T)-w(T,S^c)+\tau\vol(T),
\]
which is \eqref{eq:stable}.
For arbitrary $T\subseteq B$, comparison with $S\cup T$ and
submodularity of the cut function give
\[
 \tau\vol(T\setminus S)
 \le\cut(S\cup T)-\cut(S)
 \le\cut(T)-\cut(S\cap T).
\]
Apply this to the superlevel sets of $g^2$ and integrate using
\eqref{eq:coarea}. Their intersections with $S$ are the superlevel sets
of $(\one_Sg)^2$, proving \eqref{eq:mass-loss}.
\end{proof}

The minimization in \eqref{eq:cut-objective} is a source--sink minimum
cut, as in~\cite[Problem~7.24]{Lau}. For each $v\in B$, add a source
arc of capacity $\tau d_v$ and a sink arc of capacity $w(v,B^c)$,
and keep the internal undirected edges with their original capacities.
The cut with source-side vertices $T\subseteq B$ has capacity
$\tau\vol(B)+J(T)$.

\begin{lemma}[Boundary mass]\label{lem:boundary}
If $S$ is $\tau$-stable, then every function $g$ on $S$ satisfies
\begin{equation}\label{eq:boundary-mass}
 \sum_{v\in S}w(v,S^c)g(v)^2
             \le\tau\norm g_d^2+\sqrt{2\En_S(g)\norm g_d^2}.
\end{equation}
\end{lemma}

\begin{proof}
Integrating \eqref{eq:stable} over the superlevel sets of $g^2$ gives
\[
 \sum_{v\in S}w(v,S^c)g(v)^2
 \le\tau\norm g_d^2+
               \sum_{\{u,v\}\subset S}w_{uv}|g(u)^2-g(v)^2|.
\]
Apply \eqref{eq:variation} to the internal edges. At each vertex,
the total weight of incident internal edges is at most its original
degree, so the same norm applies.
\end{proof}

\section{Geometric partitioning of the embedding}\label{sec:geometry}

This section constructs a random partition of the embedding and bounds
the expected total mass of cells whose Gram matrices have a large sum
of eigenvalues other than the largest. We also bound the weighted cost
of edges separated by the partition. The geometric ingredient is the Euclidean
random-ball partition of Charikar, Chekuri, Goel, Guha, and
Plotkin~\cite[Section~3.1]{CCGGP}. We include the needed argument.

\begin{lemma}[Random balls]\label{lem:balls}
For every integer $q\ge1$, every finite set of points in $\R^q$,
and every $r>0$, there is a
random partition into cells of diameter at most $2r$ such that
\begin{equation}\label{eq:separation}
       \Prob[B(x)\ne B(y)]\le
                       \min\{1,2\sqrt q\norm{x-y}/r\},
\end{equation}
where $B(x)$ denotes the cell containing $x$.
\end{lemma}

\begin{proof}
The empty point set is trivial, so assume the point set is nonempty.
Sample independent centers uniformly from the union of the radius-$r$
balls centered at the given points, and assign
each point to the first sampled ball that contains it. All points are
covered almost surely, and each cell is contained in one radius-$r$
ball. The centers are sampled from the whole union of balls, not from
the finite set of given points.

Write $\mathbb B(x,r)$ for the Euclidean ball of radius $r$ centered at $x$.
For two points $x,y$, the first sampled center in
$\mathbb B(x,r)\cup \mathbb B(y,r)$ is uniform in that union. The points lie
in different cells exactly when this center belongs to the symmetric
difference. Hence
\[
 \Prob[B(x)\ne B(y)]
       =\frac{|\mathbb B(x,r)\mathbin\triangle \mathbb B(y,r)|}{|\mathbb B(x,r)\cup \mathbb B(y,r)|}.
\]
Slicing parallel to $x-y$ bounds the numerator by
$2\norm{x-y}v_{q-1}r^{q-1}$, where $v_q$ is the volume of the unit
$q$-ball. The denominator is at least $v_qr^q$. Finally,
$v_{q-1}/v_q\le\sqrt q$. For $q\ge2$, this follows from
\[
 \frac{v_q}{v_{q-1}}=
       \int_{-1}^1(1-t^2)^{(q-1)/2}\,dt\ge\frac1{\sqrt q};
\]
restrict to $|t|\le1/\sqrt q$ and use
$(1-1/q)^{(q-1)/2}\ge1/2$. The case $q=1$ is immediate.
\end{proof}

\begin{lemma}[Partitioning the spectral embedding]\label{lem:geometry}
Fix $0<\delta,\beta<1$. Let $Z:\R^\ell\to\R^k$ be an isometry
with $1\le\ell\le k$, and let $R\subseteq V$. Use the vectors
$\mathbf y_v=Z^T\mathbf f_v$ from \eqref{eq:current-field}.
There is a random partition of
$\{v\in R:\mathbf y_v\ne0\}$ into cells $B$. Write $B(v)$ for the cell containing $v$ and
$\Sigma_B=\sum_{v\in B}d_v\mathbf y_v\mathbf y_v^T$. Then
\begin{equation}\label{eq:bad-mass}
       \Ex\sum_{B:\tail(\Sigma_B)>\delta}\tr\Sigma_B\le\beta\ell,
\end{equation}
and
\begin{equation}\label{eq:partition-cost}
 \Ex\sum_{\substack{\{u,v\}\subset R:\ \mathbf y_u,\mathbf y_v\ne0\\
                                  B(u)\ne B(v)}}
 w_{uv}\bigl(\norm{\mathbf y_u}^2+\norm{\mathbf y_v}^2\bigr)
                  \le C_{\delta,\beta}\ell\sqrt{\lambda_k\log(k+1)}.
\end{equation}
\end{lemma}

\begin{proof}
Put $m_v=d_v\norm{\mathbf y_v}^2$ for every vertex and
$a_v=\mathbf y_v/\norm{\mathbf y_v}$ for nonzero rows.
Sums involving $a_v$ omit zero rows.
Global orthonormality gives
\begin{equation}\label{eq:isotropy}
                   \sum_vm_va_va_v^T=I_\ell,
                         \qquad\sum_vm_v=\ell.
\end{equation}
Set $\eta=\sqrt{\delta/[2(1+\delta)]}$, put $r=\eta/20$, and choose
\[
                   q=\left\lceil\max\{1,\log(2k/(\delta\beta))\}\right\rceil.
\]
Let $\Gamma:\R^\ell\to\R^q$ have independent $\mathcal N(0,1/q)$ entries.
Apply Lemma~\ref{lem:balls} to the unnormalized projected points
$\Gamma a_v$ for $v\in R$ with $\mathbf y_v\ne0$.
For fixed $\delta,\beta$, this choice gives
$q=O_{\delta,\beta}(\log(k+1))$.

\paragraph{Eigenvalues of the cell Gram matrices.}
If $\norm{a_u-a_v}>\eta$ and $u,v$ belong to one cell, then
$\norm{\Gamma(a_u-a_v)}\le2r<\norm{a_u-a_v}/10$.
For fixed $x\ne0$, the random variable
$q\norm{\Gamma x}^2/\norm x^2$ has the $\chi_q^2$ distribution.
The lower-tail Chernoff bound gives
\begin{equation}\label{eq:small-ball}
                    \Prob[\norm{\Gamma x}\le\norm x/10]\le e^{-q}.
\end{equation}
For completeness, writing $\chi_q^2$ for a random variable with this
distribution, Markov's inequality gives
\[
 \Prob[\chi_q^2\le q/100]
       \le e^{0.495q}\Ex e^{-49.5\chi_q^2}
       =e^{0.495q}100^{-q/2}\le e^{-q}.
\]

For a cell $B$, define the total weight of ordered pairs at distance
greater than $\eta$ by
\[
                 W_B=\sum_{\substack{u,v\in B\\
                                      \norm{a_u-a_v}>\eta}}m_um_v.
\]
It follows from \eqref{eq:small-ball} that
\begin{equation}\label{eq:far-mass}
                         \Ex\sum_BW_B\le e^{-q}\ell^2.
\end{equation}
Call a cell $B$ bad if $\tail(\Sigma_B)>\delta$, and good otherwise.
Write $m_B=\tr\Sigma_B=\sum_{v\in B}m_v$. Since
$\Sigma_B\preceq I_\ell$ by \eqref{eq:isotropy}, we have $m_B\le1+\tail(\Sigma_B)$.
Also,
\[
 m_B \tail(\Sigma_B)\le m_B^2-\tr(\Sigma_B^2)
       =\sum_{u,v\in B}m_um_v(1-\ip{a_u}{a_v}^2)
       \le\eta^2m_B^2+W_B.
\]
The first inequality uses
$\tr(\Sigma_B^2)\le\lambda_{\max}(\Sigma_B)\tr\Sigma_B$.
For the last inequality, use
$1-\ip{a_u}{a_v}^2\le\norm{a_u-a_v}^2$ for pairs at distance at
most $\eta$, and $1-\ip{a_u}{a_v}^2\le1$ for the remaining pairs.
If $\tail(\Sigma_B)>\delta$, then $\eta^2m_B\le\eta^2(1+\tail(\Sigma_B))\le \tail(\Sigma_B)/2$.
Thus $W_B\ge\delta m_B/2$ for every bad cell. Summing and using
\eqref{eq:far-mass} proves
\[
 \Ex\sum_{B:\tail(\Sigma_B)>\delta}\tr\Sigma_B
                  \le\frac2\delta e^{-q}\ell^2\le\beta\ell.
\]

\paragraph{Boundary cost.}
Conditional on $\Gamma$, Lemma~\ref{lem:balls} gives
\[
 \Prob[B(u)\ne B(v)\mid\Gamma]
                    \le\frac{2\sqrt{q}}{r}\norm{\Gamma(a_u-a_v)}.
\]
Since $\Ex\norm{\Gamma x}\le\norm x$, averaging over $\Gamma$ bounds the
separation probability by $(2\sqrt{q}/r)\norm{a_u-a_v}$.
For unit vectors $a,b$ and $s,t\ge0$,
\begin{equation}\label{eq:row-comparison}
              (s^2+t^2)\norm{a-b}\le2(s+t)\norm{sa-tb}.
\end{equation}
Indeed, the triangle inequality gives
\[
 (s+t)\norm{a-b}\le\norm{sa-tb}+\norm{ta-sb}=2\norm{sa-tb},
\]
where the equality follows from $\norm a=\norm b=1$.
Multiplying by $s+t$ and using $s^2+t^2\le(s+t)^2$ proves
\eqref{eq:row-comparison}.
Therefore the expectation in \eqref{eq:partition-cost} is at most
\[
 \frac{4\sqrt{q}}{r}\sum_{\{u,v\}}w_{uv}
       (\norm{\mathbf y_u}+\norm{\mathbf y_v})
                              \norm{\mathbf y_u-\mathbf y_v}\le\frac{4\sqrt{q}}{r}\sqrt{2\ell\,\En(FZ)}
 \le\frac{4\sqrt{2q}}{r}\ell\sqrt{\lambda_k}.
\]
Here Cauchy--Schwarz uses
$\sum_{\{u,v\}}w_{uv}(\norm{\mathbf y_u}+\norm{\mathbf y_v})^2
\le2\ell$. This proves the desired bound.
\end{proof}

Discarding whole cells with $\tail(\Sigma_B)>\delta$ creates no new
boundary inside a retained cell. The vertices of the discarded cells
remain available for later stages. We do not discard individual
vertices according to the norm of their projected vectors.

\section{Adaptive selection of vertex sets}\label{sec:adaptive}

For the construction, assume $k\ge2$ and $\lambda_k>0$.
The remaining cases are handled in the proof of Theorem~\ref{thm:main}.
Throughout the construction, $G$ and $k$ are fixed. Set
\begin{equation}\label{eq:constants}
 \varepsilon=\frac1{64},\qquad \delta=\frac{\varepsilon}{4},\qquad \beta=\varepsilon,
 \qquad \tau=C_0\sqrt{\lambda_k\log(k+1)},
\end{equation}
where the absolute constant $C_0$ is chosen in the proof of
Lemma~\ref{lem:stage} below. The threshold $\tau$ depends on $G$ and $k$
but remains fixed throughout the construction.

\subsection{One selection step}

\begin{lemma}\label{lem:stage}
Suppose previously selected sets are pairwise disjoint and
$\tau$-stable, and let $R$ be the complement of their union.
Let $Z:\R^\ell\to\R^k$
be an isometry, where $1\le\ell\le k$, such that
\begin{equation}\label{eq:active-retention}
                           Z^TF^*\one_RFZ\succeq(1-\varepsilon)I_\ell.
\end{equation}
One can select additional disjoint nonempty sets $S_i\subseteq R$,
each $\tau$-stable and satisfying
\[
                        \tail(Z^TF^*\one_{S_i}FZ)\le\delta,
\]
so that, for $R'=R\setminus\bigcup_iS_i$,
\begin{equation}\label{eq:trace-reduction}
                            \tr(Z^TF^*\one_{R'}FZ)\le\ell/4.
\end{equation}
\end{lemma}

\begin{proof}
Put $\alpha(v)=\norm{\mathbf y_v}$, where $\mathbf y_v=Z^T\mathbf f_v$.
Then $\norm\alpha_d^2=\ell$. The reverse triangle inequality gives
$|\alpha(u)-\alpha(v)|\le\norm{\mathbf y_u-\mathbf y_v}$ on every edge,
and hence $\En(\alpha)\le\En(FZ)\le\ell\lambda_k$.
Equation~\eqref{eq:active-retention} implies
\begin{equation}\label{eq:old-mass}
                     \sum_{v\notin R}d_v\alpha(v)^2
                                                    \le\varepsilon\ell.
\end{equation}

Apply Lemma~\ref{lem:boundary} to $\alpha$ on each previously selected
set $S$. Their internal edge sets are disjoint, so Cauchy--Schwarz gives
\[
 \sum_S\sqrt{2\En_S(\alpha)\norm{\one_S\alpha}_d^2}
 \le\sqrt{2\Bigl(\sum_S\En_S(\alpha)\Bigr)
                 \Bigl(\sum_S\norm{\one_S\alpha}_d^2\Bigr)}
 \le\ell\sqrt{2\varepsilon\lambda_k},
\]
where the sums run over the previously selected sets, and the last
inequality uses \eqref{eq:old-mass} and $\En(\alpha)\le\ell\lambda_k$.
Summing \eqref{eq:boundary-mass} therefore gives
\begin{equation}\label{eq:old-boundary}
 \sum_{u\notin R,v\in R}w_{uv}\alpha(u)^2
                 \le\tau\varepsilon\ell+\ell\sqrt{2\varepsilon\lambda_k}.
\end{equation}
The sum on the left of \eqref{eq:old-boundary} is no larger than the sum of the
left sides of \eqref{eq:boundary-mass} over the previously selected
sets; edges between two such sets may also appear in the latter sum.

For an edge from $v\in R$ to $u\notin R$, the inequality
$\alpha(v)^2\le|\alpha(v)^2-\alpha(u)^2|+\alpha(u)^2$
shows that
\begin{equation}\label{eq:restrict-variation}
                 \TV(\one_R\alpha^2)\le\TV(\alpha^2)+\sum_{u\notin R,v\in R}w_{uv}\alpha(u)^2.
\end{equation}
By \eqref{eq:variation}, $\TV(\alpha^2)\le\ell\sqrt{2\lambda_k}$.

Apply Lemma~\ref{lem:geometry} to the nonzero rows in $R$.
Placing the endpoints of an edge inside $R$ in different cells increases its contribution
to the sum of variations by at most
$w_{uv}(\alpha(u)^2+\alpha(v)^2)$. If either row is zero, there is
no increase. Hence
\begin{equation}\label{eq:raw-variation}
 \Ex\sum_B\TV(\one_B\alpha^2)
 \le\ell\left(\tau\varepsilon+\sqrt{2\varepsilon\lambda_k}
                 +\sqrt{2\lambda_k}+C_{\delta,\beta}\sqrt{\lambda_k\log(k+1)}\right).
\end{equation}
Discard whole cells whose Gram matrices satisfy
$\tail(\Sigma_B)>\delta$.
Their expected mass is at most $\beta\ell$. In each remaining cell
$B$, apply Lemma~\ref{lem:cut} and let $S_B$ be its selected subset.
Keep the nonempty sets $S_B$. Each is $\tau$-stable, and its Gram
matrix satisfies $\tail(Z^TF^*\one_{S_B}FZ)\le\delta$ by \eqref{eq:tail}.

The vertices remaining in $R'$ are those with zero rows, those in bad
cells, and those rejected inside good cells. Applying \eqref{eq:mass-loss}
with $g=\one_B\alpha$ bounds the mass of the last group.
Thus, using \eqref{eq:raw-variation} and \eqref{eq:constants},
\[
 \Ex\tr(Z^TF^*\one_{R'}FZ)
 \le\beta\ell+\frac1{\tau}\Ex\sum_{B\text{ good}}\TV(\one_B\alpha^2)
 \le\ell\left(\beta+\varepsilon+\frac{C_{\delta,\beta}}{C_0}
       +\frac{\sqrt{2\varepsilon}+\sqrt2}{C_0\sqrt{\log(k+1)}}\right).
\]
Since $\beta+\varepsilon=1/32$ and $\log(k+1)\ge\log3$, it suffices to choose
\[
 C_0\ge\frac{32}{7}\left(C_{\delta,\beta}
              +\frac{\sqrt{2\varepsilon}+\sqrt2}{\sqrt{\log3}}\right).
\]
Here $C_{\delta,\beta}$ is the constant in \eqref{eq:partition-cost};
the fixed values of $\delta,\beta,\varepsilon$ make this choice
independent of the stage and of $k$. With this choice, the expectation
above is at most $(1/32+7/32)\ell=\ell/4$.
Some outcome therefore satisfies
\eqref{eq:trace-reduction}. All other assertions hold for the same
outcome. No independence between successive selection steps is needed.
\end{proof}

\subsection{The sum of the Gram matrices}

Start with $R_0=V$ and the coefficient projector $\Pi_0=I_k$.
At step $j$, choose an orthonormal basis of $\ran\Pi_{j-1}$ and
apply Lemma~\ref{lem:stage} to $R_{j-1}$ using the corresponding
isometry. Let $R_j$ be the vertices of $R_{j-1}$ left after removing
the newly selected sets. Define the Gram matrices before and after
this selection, in the same coefficient space, by
\[
 K_j=\Pi_{j-1}F^*\one_{R_{j-1}}F\Pi_{j-1},\qquad
 Q_j=\Pi_{j-1}F^*\one_{R_j}F\Pi_{j-1}.
\]
The newly selected sets partition $R_{j-1}\setminus R_j$, and
\begin{equation}\label{eq:exact-split}
 K_j-Q_j=\Pi_{j-1}F^*\one_{R_{j-1}\setminus R_j}F\Pi_{j-1}\succeq0.
\end{equation}
Let $\Pi_j$ be the spectral projector of $Q_j$, within
$\ran\Pi_{j-1}$, for eigenvalues strictly greater than $1-\varepsilon$.
Then $\Pi_j\preceq\Pi_{j-1}$, and any isometry onto
$\ran\Pi_j$ satisfies \eqref{eq:active-retention} with $R=R_j$.
If $\ell_j=\rank\Pi_j$, then
\begin{equation}\label{eq:rank-reduction}
 \ell_j\le\frac{\tr Q_j}{1-\varepsilon}
                   \le\frac{\ell_{j-1}}{4(1-\varepsilon)}.
\end{equation}
Let $N$ be the first step for which $\Pi_N=0$.
Iterating \eqref{eq:rank-reduction} gives
\[
 \ell_j\le k[4(1-\varepsilon)]^{-j},\qquad
 N\le1+\left\lfloor\frac{\log k}{\log(4(1-\varepsilon))}\right\rfloor
       =O(\log(k+1)),
\]
since a nonnegative integer rank smaller than one is zero.

For each selected set $S_i$ produced at step $j$, define the localized map
$F_i:\R^k\to L^2(d)$ by
\begin{equation}\label{eq:localized-fields}
                       F_ix=\one_{S_i}F\Pi_{j-1}x.
\end{equation}
The maps have pairwise disjoint supports. Their weighted Gram matrices
satisfy $\tail(F_i^*F_i)\le\delta$: passing from the current coefficient
space to $\R^k$ adds only zero eigenvalues to the Gram matrix of a selected set.
By \eqref{eq:exact-split}, summing $F_i^*F_i$ over the sets selected at
step $j$ gives $K_j-Q_j$.

\begin{lemma}\label{lem:covariance}
The localized maps satisfy
\begin{equation}\label{eq:total-covariance}
                          \varepsilon I_k\preceq\sum_iF_i^*F_i\preceq I_k.
\end{equation}
\end{lemma}

\begin{proof}
Put $P_j=\Pi_{j-1}-\Pi_j$ and set $K_{N+1}=0$.
Since the projectors are nested, $\Pi_j\Pi_{j-1}=\Pi_j$.
Also, $\Pi_j$ commutes with $Q_j$ because it is a spectral projector of $Q_j$. Thus
\[
 \Pi_jQ_j\Pi_j=\Pi_jF^*\one_{R_j}F\Pi_j=K_{j+1},\qquad
 \Pi_jQ_jP_j=Q_j\Pi_jP_j=0.
\]
By symmetry the other off-diagonal block also vanishes. The eigenvalues
of $Q_j$ on $\ran P_j$ are at most $1-\varepsilon$, so
\[
 Q_j=K_{j+1}+P_jQ_jP_j,\qquad
 0\preceq P_jQ_jP_j\preceq(1-\varepsilon)P_j.
\]
Since $K_1=I_k$, summing this identity gives
\[
 \sum_iF_i^*F_i=\sum_{j=1}^{N}(K_j-Q_j)
                        =I_k-\sum_{j=1}^{N}P_jQ_jP_j.
\]
The projectors $P_j$ have mutually orthogonal ranges and sum to $I_k$.
Consequently $0\preceq\sum_jP_jQ_jP_j\preceq(1-\varepsilon)I_k$,
which proves \eqref{eq:total-covariance}.
\end{proof}

The only commutation used above is that of $Q_j$ with its own spectral
projector; Gram matrices of different cells need not commute.

\section{Internal energy and exactly \texorpdfstring{$k$}{k} sets}
\label{sec:extraction}

The internal edges of the selected sets are disjoint. This allows a
single maximal estimate to control all coefficient projections.

\begin{lemma}[A dyadic maximal estimate]\label{lem:maximal}
Let $I_k=\Pi_0\succeq\cdots\succeq\Pi_N=0$ be nested orthogonal
projectors, and let $\mathcal T:\R^k\to L^2(\Omega,\mu)$ be a bounded linear map.
Put $s_N=2+\lceil\log_2(N+1)\rceil$. Then, for every $x\in\R^k$,
\begin{equation}\label{eq:maximal}
              \int\max_{0\le j<N}|\mathcal T\Pi_jx|^2\,d\mu
                                     \le s_N^2\norm{\mathcal T}^2\norm x^2.
\end{equation}
\end{lemma}

\begin{proof}
The differences $P_j=\Pi_{j-1}-\Pi_j$ are mutually orthogonal
projectors. Pad the sequence with zero projectors to a power-of-two
length. Each tail sum $\Pi_j=\sum_{h>j}P_h$ can be written as a sum
of at most $s_N$ operators $P_I=\sum_{h\in I}P_h$, where the index
sets $I$ are disjoint dyadic intervals. Pointwise Cauchy--Schwarz gives
\[
 \max_j|\mathcal T\Pi_jx|^2
    \le s_N\sum_{I\text{ dyadic}}|\mathcal T P_Ix|^2.
\]
At each dyadic level the interval projectors have orthogonal ranges,
so
\[
             \sum_{I\text{ at that level}}\norm{\mathcal T P_Ix}_{L^2}^2
                      \le \norm{\mathcal T}^2\sum_I\norm{P_Ix}^2=\norm{\mathcal T}^2\norm x^2.
\]
There are at most $s_N$ levels. Integrating proves
\eqref{eq:maximal}. Orthogonality is needed in coefficient space;
the functions $\mathcal T P_jx$ need not be orthogonal.
\end{proof}

The classical Rademacher--Menshov maximal inequality controls maximal
partial sums of orthogonal functions; see~\cite[Proposition~2.1]{Meaney}.
The proof above adapts the usual dyadic argument to a bounded operator
on coefficient space and makes the dependence on $N$ explicit.

\begin{lemma}\label{lem:energy}
For every $x\in\R^k$, the maps in \eqref{eq:localized-fields} satisfy
\begin{equation}\label{eq:internal-energy}
                           \sum_i\En_{S_i}(F_ix)
                                        \le s_N^2\lambda_k\norm x^2.
\end{equation}
\end{lemma}

\begin{proof}
Fix an orientation of each edge and apply Lemma~\ref{lem:maximal}
to the edge-gradient map
\[
       (\mathcal T x)(\{u,v\})=\sqrt{w_{uv}}\ip{\mathbf f_u-\mathbf f_v}{x},
\]
with counting measure on edges. Equation~\eqref{eq:eigenfield} gives
$\norm{\mathcal T}^2\le\lambda_k$. Every edge lies inside at most one selected set.
Its contribution to the left side of
\eqref{eq:internal-energy} is therefore at most the corresponding
pointwise maximum in \eqref{eq:maximal}.
\end{proof}

The final lemma converts the bounds on the Gram matrices and internal
energies into exactly $k$ sets. We use the dimension argument
of~\cite{Li}, with each local eigenvalue count truncated at two,
and apply it to the degree-preserving local operators instead of
Dirichlet operators.

\begin{lemma}[Counting local eigenvalues]\label{lem:capacity}
Let $\tau\ge0$, and let $F_i:\R^k\to L^2(d)$ be linear maps supported
on pairwise disjoint nonempty sets $S_i$. Suppose each $S_i$ is $\tau$-stable
and that, for some $0\le\delta<\gamma$ and $\Lambda\ge0$,
\[
 \sum_iF_i^*F_i\succeq\gamma I_k,\qquad
 \tail(F_i^*F_i)\le\delta,\qquad
 \sum_i\En_{S_i}(F_ix)\le \Lambda\norm x^2\quad(x\in\R^k).
\]
Then
\begin{equation}\label{eq:capacity-bound}
                         \varphi_k(G)\le\tau+4\sqrt{\Lambda/(\gamma-\delta)}.
\end{equation}
\end{lemma}

\begin{proof}
First assume $\Lambda>0$ and put $\sigma=2\Lambda/(\gamma-\delta)$.
Let $c_i$ be the number of local eigenvalues $\lambda_j(S_i)$ at most $\sigma$,
truncated at two. Since $\lambda_1(S_i)=0$, every $c_i$
is at least one. We claim that $\sum_ic_i\ge k$.

Suppose instead that the sum is less than $k$, and let $p$ be the number
of sets. Since each $c_i$ is either one or two, $p<k$ and fewer than
$k-p$ indices satisfy $c_i=2$.
For each $i$ with $c_i=1$, impose on $x$ the linear constraint that $F_ix$
be orthogonal, in $L^2(d|_{S_i})$, to a unit vector spanning the
one-dimensional local spectral subspace for eigenvalues at most $\sigma$.
For each $i$ with $c_i=2$, choose a unit eigenvector $u_i$ corresponding
to the largest eigenvalue of $F_i^*F_i$, and require that $x$ be
orthogonal to $u_i$. These are $p$
homogeneous linear constraints. Their common kernel contains
$k-p$ orthonormal vectors $x_1,\ldots,x_{k-p}$.

For $c_i=1$, the spectral theorem gives
$\norm{F_ix_a}_d^2\le\En_{S_i}(F_ix_a)/\sigma$.
For $c_i=2$, the matrix $\sum_a x_ax_a^T$ is the orthogonal projector
onto the span of the vectors $x_a$. Each $x_a$ is orthogonal to $u_i$, so
$0\preceq\sum_a x_ax_a^T\preceq I_k-u_iu_i^T$. Consequently,
\[
 \sum_a\norm{F_ix_a}_d^2
 =\tr\Bigl(F_i^*F_i\sum_a x_ax_a^T\Bigr)
 \le\tr\bigl(F_i^*F_i(I_k-u_iu_i^T)\bigr)
 =\tail(F_i^*F_i)\le\delta.
\]
Thus $\delta$ bounds the sum of the squared $L^2(d)$-norms of all
$k-p$ functions $F_ix_a$.
Summing over $i$ and $a$ and using the Gram matrix and energy
hypotheses gives
\[
 \gamma(k-p)\le\frac{\Lambda(k-p)}{\sigma}+\delta(k-p)
            =\frac{\gamma+\delta}{2}(k-p)<\gamma(k-p),
\]
a contradiction.

For a set contributing one to the count, use $S_i$ itself; its
conductance is at most $\tau$. For one contributing two,
$\lambda_2(S_i)\le \sigma$. If the local graph is disconnected, use two
of its connected components, each of local conductance zero.
Otherwise, a second eigenfunction $f$
is orthogonal to the constant function and has two nonzero sign parts.
The pointwise contraction inequality gives
\[
 \En_{S_i}(f_+)\le\ip{f_+}{\Delta_{S_i}f}_d
      =\lambda_2(S_i)\norm{f_+}_d^2.
\]
Applying the same estimate to $-f$ gives the bound for $f_-$.
Cheeger rounding in the local graph gives two disjoint
nonempty sets of local conductance at most $\sqrt{2\sigma}$.
Equation~\eqref{eq:local-global} bounds their original conductance
by $\tau+2\sqrt{2\sigma}$, which is \eqref{eq:capacity-bound}.
Since $\sum_i c_i\ge k$, retain any $k$ of the
resulting disjoint sets. If $\Lambda=0$, apply the positive-$\Lambda$ case with
any $\Lambda'>0$, for which the energy hypothesis still holds. Let $\Lambda'$ tend
to zero and pass to the limit over the finite family of possible sets.
\end{proof}

\begin{proof}[Proof of Theorem~\ref{thm:main}]
For $k=1$, take $V$. If $\lambda_k=0$, the graph has at least $k$
connected components, each of positive volume; use any $k$ of them.
Otherwise apply the construction of Section~\ref{sec:adaptive}.
Lemmas~\ref{lem:covariance} and~\ref{lem:energy} verify the
hypotheses of Lemma~\ref{lem:capacity} with the threshold $\tau$
from \eqref{eq:constants}, $\gamma=\varepsilon$, $\delta=\varepsilon/4$,
and $\Lambda=s_N^2\lambda_k$.
Thus
\[
       \varphi_k(G)\le C_0\sqrt{\lambda_k\log(k+1)}
                           +\frac{4s_N\sqrt{\lambda_k}}{\sqrt{\varepsilon-\delta}}.
\]
By \eqref{eq:rank-reduction}, $N=O(\log(k+1))$. Consequently
\[
       s_N=O(1+\log\log(k+2))=O(\sqrt{\log(k+1)}),
\]
which proves \eqref{eq:main}.
\end{proof}

\section{Sharpness of the logarithmic factor}\label{sec:sharpness}

We recall the noisy-hypercube example from~\cite[Section~4.4]{LOT}
to check that the lower bound uses the same spectral index.

\begin{proposition}\label{prop:sharpness}
For every integer $k\ge3$, there is a weighted graph $G$ such that
\[
                  \frac{\varphi_k(G)}{\sqrt{\lambda_k(G)}}
                                  \ge\sqrt{\frac{\log k}{8\log2}}.
\]
\end{proposition}

\begin{proof}
Fix $k\ge3$. Let $G$ be the graph on $\{0,1\}^k$ with
weights $w_{xy}=\zeta^{|x-y|}$, including loops,
where $|x-y|$ is Hamming distance and $\zeta=(\log2)/\log k$.
The common degree is $(1+\zeta)^k$. The transition operator is
the Boolean noise operator $T_{\rho}$ with correlation
$\rho=(1-\zeta)/(1+\zeta)$. Its eigenvalue on each Walsh
character of degree $j$ is $\rho^j$.
Thus the first nonzero Laplacian eigenvalue has multiplicity $k$ and
\[
                         \lambda_k(G)=\frac{2\zeta}{1+\zeta}.
\]

Let $\mu$ be the uniform probability measure. For every nonempty
$S\subseteq\{0,1\}^k$, the hypercontractive estimate used
in~\cite[Theorem~4.12]{LOT} gives
\[
 \frac{\ip{\one_S}{T_{\rho}\one_S}_{\mu}}{\mu(S)}
       \le\mu(S)^{(1-\rho)/(1+\rho)}=\mu(S)^{\zeta}.
\]
Equivalently, $\varphi(S)\ge1-\mu(S)^{\zeta}$.
Indeed, hypercontractivity gives
\[
 \ip{\one_S}{T_{\rho}\one_S}_\mu
       =\norm{T_{\sqrt{\rho}}\one_S}_{L^2(\mu)}^2
       \le\norm{\one_S}_{L^{1+\rho}(\mu)}^2.
\]
Among $k$ disjoint nonempty sets, one has measure at most $1/k$, and therefore
conductance at least $1-k^{-\zeta}=1/2$. It follows that
\[
 \frac{\varphi_k(G)}{\sqrt{\lambda_k(G)}}
       \ge\frac12\sqrt{\frac{1+\zeta}{2\zeta}}
       \ge\sqrt{\frac{\log k}{8\log2}}.
\]
\end{proof}

\section*{Acknowledgments}
The author thanks Lap Chi Lau for bringing Theorem~7.17 of his book
manuscript to the author's attention after the first version
of~\cite{Li} was posted. The exposition of cut improvement and local
conductance in that manuscript guided the present argument.
Lau subsequently informed the author, in private correspondence, that
he had independently obtained the same optimal bound about two weeks
before the present manuscript was written.

\section*{Disclosure of AI assistance}
OpenAI's ChatGPT and Codex (GPT-6) were used to develop the proof and
prepare the manuscript. The human author takes full responsibility
for the mathematical arguments and the final manuscript.

\end{document}